\documentclass[12pt]{article}
\usepackage{amsmath, amssymb, amsfonts, amstext, amsthm, textcomp, enumerate}
\usepackage[mathscr]{euscript}
\newtheorem{thm}{Theorem}[section]
\newtheorem{rem}[thm]{Remark}
\newtheorem{ex}[thm]{Example}

\begin{document}

\begin{center}
\vspace{3cm}
 {\bf \large When is $(A+B)^{\dagger}=A^{\dagger}+B^{\dagger}$?}\\
\vspace{3cm}

{\bf K.C. Sivakumar}\\
Department of Mathematics\\
Indian Institute of Technology Madras\\
Chennai 600 036, India. \\
\end{center}

\begin{abstract}
We address the question as to when it is true that $(A+B)^{\dagger}=A^{\dagger}+B^{\dagger},$ where $\dagger$ denotes the Moore-Penrose inverse. A similar question is addressed for the group inverse.

\end{abstract}



\newpage
\section{Introduction}
The motivation for this short note is the work of \cite{bomuhl}, where the authors seek to solve the equation $\frac{1}{a+b}=\frac{1}{a}+\frac{1}{b}$ not only for reals or complex numbers, but also for matrices with real entries; in the first case there is no solution, while in the second and third cases, it is shown that there are infinitely many solutions. Specifically, they show that $(A+B)^{-1}=A^{-1}+B^{-1}$ holds for real matrices of order $n$, if and only if $n$ is even and describe a method of constructing such matrices. Recall that a real vector space $V$ is said to have a complex structure if there is a linear operator $J$ on $V$ such that $J^2=-I$. It is easy to observe that a finite dimensional real vector space admists a complex structure if and only if its dimension is even. In \cite{dan}, it is shown that the identity above holds in a finite dimensional vector space $V$ if and only if $V$ admits a complex a structure, thereby obtaining the same conclusion as in \cite{bomuhl}, as a consequence. Here, we consider a more general question of asking when the identity above extends to generalized inverses. More precisely, we present {\it sufficient conditions} on (possibly) rectangular matrices $A$ and $B$ with complex entries such that the equation $(A+B)^{\dagger}=A^{\dagger}+B^{\dagger}$ holds, where $\dagger$ stands for the Moore-Penrose inverse (see Remark \ref{where}). We also consider the case of the group inverse. 

\section{Preliminaries}
The symbol $\mathbb{C}^{m \times n}$ denotes the set of all complex matrices of order $m \times n$. For $A \in \mathbb{C}^{m \times n},$ we use $R(A)$ to denote its range space and $N(A)$ to denote its null space. For any matrix $X$ with complex entries, $X^*$ denotes the conjugate transpose. Let us recall that for $A \in \mathbb{C}^{m \times n},$ the Moore-Penrose (generalized) inverse (or the pseudo inverse) of $A$, denoted by $A^{\dagger}$ is the unique matrix $X \in \mathbb{C}^{n \times m}$ that satisfies the equations $AXA=A, XAX=X, (AX)^*=AX$ and $(XA)^*=XA$. One of the many ways of showing the existence of the Moore-Penrose inverse is by using the full-rank factorization. A matrix $A \in \mathbb{C}^{m \times n}$ is said to have a full-rank factorization if there exist $F \in \mathbb{C}^{m \times r}$ and $G \in \mathbb{C}^{r \times n}$ such that $rank(F)=rank(G)=rank(A)=r$ and $A=FG$. It then follows that $A^{\dagger}=G^*(GG^*)^{-1}(F^*F)^{-1}F^*$. In fact, in this case, one has $F^{\dagger}=(F^*F)^{-1}F^*$ and $G^{\dagger}=G^*(GG^*)^{-1},$ so that $F^{\dagger}$ is a left inverse of $F$, while $G^{\dagger}$ is a right inverse of $G$. More generally, for any matrix $X$, one has the formulae: $X^{\dagger}=(X^*X)^{\dagger}X^*=X^*(XX^*)^{\dagger}$. The following properties will be frequently used: $R(A)=R(AA^{\dagger})$ and $R(A^*)=A^{\dagger}A=R(A^{\dagger})$. In particular, it follows that if $x \in R(A)$, then $AA^{\dagger}x=x$, which may be extended to the idea that if $R(B) \subseteq R(A)$ then $AA^{\dagger}B=B$. Similarly, $A^{\dagger}A$ acts like identity on $R(A^*)$. 

For the reader who is encountering the Moore-Penrose inverse for the first time, here is a motivation: For the linear system $Ax=b$, given $A \in \mathbb{C}^{n \times n}$ and $b \in \mathbb{C}^n$ with $A$ nonsingular, one has $x=A^{-1}b$ as the unique solution. Now, consider the system $Ax=b$, given $A \in \mathbb{C}^{m \times n}$ and $b \in \mathbb{C}^m$. Set $x^0=A^{\dagger}b$. If the system has a unique solution, then $x^0$ is that unique solution; if it has infinitely many solutions, then $x^0$ is the solution that has the additional property that it has the least (euclidean) norm, among all the solutions; if the system does not have a solution, but has a unique least squares solution, then $x^0$ is that solution, and finally, if the system is not consistent and has infinitely many least squares solution, then $x^0$ is the unique least squares solution with the least norm. 

Another generalized inverse, this time for square matrices, is recalled next. Let $A \in \mathbb{C}^{n \times n}$. If there exists $X \in \mathbb{C}^{n \times n}$ such that $AXA=A, XAX=X$ and $AX=XA$, then such an $X$ must be unique and is referred to as the group inverse of $A$. It is denoted by $A^{\#}$. A necessary and sufficient condition for the existence of the group inverse is the condition that $rank(A^2)=rank(A)$, which is of course, the same as $R(A^2)=R(A)$, which in turn, is equivalent to the condition $N(A^2)=N(A)$. The nomenclature for the group inverse comes from the fact that the set consisting of $A$ and its positive powers, $A^{\#}$ and its positive powers, forms a group under matrix multiplication, where $AA^{\#}$ is the identity element and $A^{\#}$ is the inverse of $A$. It is useful to note that the group inverse of $A$, if it exists, is a polynomial in $A$. Once again, a formula for the group inverse may be given in terms of a full-rank factorization: if $A=FG$ is a full-rank factorization, then $A^{\#}$ exists if and only if $GF$ is invertibe. In that case, one also has $A^{\#}=F(GF)^{-2}G.$ Analogous to the Moore-Penrose inverse, one has: $R(A)=R(AA^{\#})=R(A^{\#})$. If $R(B) \subseteq R(A)$ then $AA^{\#}B=B$. 

It may be emphasized that while the Moore-Penrose inverse exists for all matrices, the group inverse of a given matrix need not exist. For instance, no nilpotent matrix possesses the group inverse. Of course, if $A$ is square and nonsingular, then one has $A^{\dagger}=A^{\#}=A^{-1}$.

The following result will also be useful.

\begin{thm}\cite{bengre}\label{grp}
For $C \in \mathbb{C}^{n \times n}$, let $X$ be a matrix satisfying $$Xp=0 \Longleftrightarrow Cp=0$$
and $$Xp=q \Longleftrightarrow Cq=p \ \ for ~all \ \ p,q \in R(C). $$ 
Then $X=C^{\#}.$
\end{thm}

For the reader interested in studying applications of the group inverse, we point to \cite{mey}, where a probabilistic interpretation for the group inverse of a matrix arising from a Markov chain, is presented. For more details and proofs of the facts on generalized inverses that are used here, we refer the reader to the excellent book \cite{bengre}. 

\section{The case of the Moore-Penrose inverse}
First, we collect some prelminary properties.

\begin{thm}\label{prelim}
Let $A,B \in \mathbb{C}^{n \times n}$. Suppose that 
$$AB^*+BB^*=0 \ \ and \ \ B^*A+B^*B=0.$$
We then have:\\
$(a)$ $AB^{\dagger}+BB^{\dagger}=0 \ \ and \ \ B^{\dagger}A+B^{\dagger}B=0.$\\
$(b)$ $N(A) \subseteq N(B) \ \ and \ \ N(A^*) \subseteq N(B^*)$.\\
$(c)$ $BA^{\dagger}+BB^{\dagger}=0 \ \ and \ \ A^{\dagger}B+B^{\dagger}B=0.$\\
$(d)$ $BA^{\dagger}A=AA^{\dagger}B=B \ \ and \ \ BA^{\dagger}B=-B.$\\
$(e)$ $A^{\dagger}BB^{\dagger}=-B, B^{\dagger}AA^{\dagger}=B^{\dagger} \ \ and \ \ A^{\dagger}BA^{\dagger}+B^{\dagger}BA^{\dagger}=0$.\\
$(f)$ $BA^{\dagger}$ and $A^{\dagger}B$ are hermitian. 
\end{thm}
\begin{proof}
$(a)$ We have $0=AB^*+BB^*=(A+B)B^*$ and so, $$AB^{\dagger}+BB^{\dagger}=(A+B)B^{\dagger}=(A+B)B^*(BB^*)^{\dagger}=0.$$ Also, $0=B^*A+B^*B=B^*(A+B)$, which yields $$B^{\dagger}(A+B)=(B^*B)^{\dagger}B^*(A+B)=0.$$
$(b)$ From the equality  $B^*A+B^*B=0,$ it follows that if $x \in N(A),$ then $B^*Bx=0$, which implies that $Bx=0$, due to the well known condition $N(B^*B)=N(B)$. Thus $N(A) \subseteq N(B)$. Taking conjugate transposes of $AB^*+BB^*=0,$ we get $BA^*+BB^*=0.$ Now, if $x \in N(A^*)$, then $BB^*x=0$, which yields $B^*x=0$, showing that $N(A^*) \subseteq N(B^*)$.\\
$(c)$ Taking the transposes of $AB^*+BB^*=0$, one obtains $0=BA^*+BB^*=B(A^*+B^*)$. Arguing as earlier, we have $0=B(A^{\dagger}+B^{\dagger}),$ yielding the first identity. The second identity follows similarly. \\
$(d)$ One has $(A^{\dagger}A)^*B^*=A^{\dagger}AB^*=B^*$, where the last equality is due to the fact that $R(B^*) \subseteq R(A^*)$ (which in turn, is due to $N(A) \subseteq N(B)$). Upon taking transposes, one obtains $B=BA^{\dagger}A$. Next, since $R(B) \subseteq R(A)$, one has $AA^{\dagger}B=B$. From the second identity of $(c)$, upon premultiplying by $B$, one has $BB^{\dagger}B+BA^{\dagger}B=0$, i.e $BA^{\dagger}B=-B$.\\
$(e)$ Using the second identity of $(c)$, we have $A^{\dagger}B=-B^{\dagger}B$. Thus, one has $A^{\dagger}BB^{\dagger}=-B^{\dagger}BB^{\dagger}=-B^{\dagger}$, proving the first part. Also, $AA^{\dagger}(B^{\dagger})^*=(B^{\dagger})^*$, since $R((B^{\dagger})^*)=R(B)$ (which is contained in $R(A)$). Upon taking transposes, we get  $B^{\dagger}AA^{\dagger}=B^{\dagger}$. Upon post multiplying the second identity of $(c)$ by $A^{\dagger}$, we obtain the third part. \\
$(f)$ By $(c)$, we have $BA^{\dagger}=-BB^{\dagger}$, proving that $BA^{\dagger}$ is hermitian. The second part is similar.
\end{proof}

\begin{thm}\label{mainmp}
Let $A,B \in \mathbb{C}^{n \times n}$ be related in such a way that 
$$AB^*+BB^*=0 \ \ and \ \ B^*A+B^*B=0.$$
We then have: $$(A+B)^{\dagger}=A^{\dagger} + B^{\dagger}  .$$
\end{thm}
\begin{proof}
Set $X=A^{\dagger} + B^{\dagger}$. We show that $X=(A+B)^{\dagger}$ by verifying the four equations for the Moore-Penrose inverse. We have $$(A+B)X=AA^{\dagger}+AB^{\dagger}+BA^{\dagger}+BB^{\dagger}=AA^{\dagger}+BA^{\dagger},$$
by using the fact that $AB^{\dagger}+BB^{\dagger}=0$. By $(f)$ of Theorem \ref{prelim}, $BA^{\dagger}$ is hermitian and so $(A+B)X$ is hermitian. Also, 
\begin{eqnarray*}
X(A+B)X&=&(A^{\dagger}+B^{\dagger})(AA^{\dagger}+BA^{\dagger})\\
&=&A^{\dagger}AA^{\dagger}+A^{\dagger}BA^{\dagger}+B^{\dagger}AA^{\dagger}+B^{\dagger}BA^{\dagger}\\
&=&A^{\dagger} + B^{\dagger}AA^{\dagger}\\
&=&A^{\dagger}+B^{\dagger},
\end{eqnarray*}
where we have made use of the second and third parts of $(e)$ of Theorem \ref{prelim}. Further, 
\begin{eqnarray*}
(A+B)X(A+B)&=&(AA^{\dagger}+BA^{\dagger})(A+B)\\
&=&AA^{\dagger}A+AA^{\dagger}B+BA^{\dagger}A+BA^{\dagger}B\\
&=&A + BA^{\dagger}A\\
&=&A+B,
\end{eqnarray*}
where we have used all the formulae in $(d)$ of Theorem \ref{prelim}. Finally, one has 
$$X(A+B)=A^{\dagger}A+A^{\dagger}B+B^{\dagger}A+B^{\dagger}B=A^{\dagger}A+A^{\dagger}B,$$
where the second part of $(a)$ of Theorem \ref{prelim} was used. Again, by $(f)$ of Theorem \ref{prelim}, since $A^{\dagger}B$ is hermitian, it follows that $X(A+B)$ is hermitian, completing the proof.
\end{proof}

\section{Group inverse analogue}
First, we collect some prelminary properties.

\begin{thm}\label{prelimgrp}
Let $A,B \in \mathbb{C}^{n \times n}$. Suppose that $B^{\#}$ exists and that one has the following relationships between $A$ and $B$:\\
$$AB^{\#}+BB^{\#}=0 \ \ and \ \ B^{\#}A+B^{\#}B=0.$$
We then have:\\
$(a)$ $N(A) \subseteq N(B) \ \ and \ \ N(A^*) \subseteq N(B^*)$.\\
$(b)$ $AB=BA=-B^2.$\\
$(c)$ $R(A+B) \subseteq N(B) \ \ and \ \ R(A^*+B^*) \subseteq N(B^*).$\\
$(d)$ $R(B^*) \subseteq N(A^*+B^*) \ \ and \ \ R(B) \subseteq N(A+B).$\\
$(e)$ $(A+B)^{\#}, A^{\#}$ and $(AB)^{\#}$ exist.
\end{thm}
\begin{proof}
$(a)$ Let $Ax=0$. Then from the second condition, one has $$0=B^{\#}Ax+B^{\#}Bx=B^{\#}Bx,$$ which upon premultiplying by $B$, gives $Bx=0$. Hence $N(A) \subseteq N(B)$. Similarly, by taking the conjugate transposes of the first condition and premultiplying by $B^*$, one has the implication $A^*x=0 \Longrightarrow B^*x=0$, showing that the second inclusion holds. \\
$(b)$ We have $$AB=AB^{\#}B^2=-BB^{\#}B^2=-B^2,$$ using the first identity. Employing the second identity, one has $$BA=B^2B^{\#}A=-B^2B^{\#}B=-B^2.$$ 
$(c)$ The second identity is the same as $B^{\#}(A+B)=0$ and this shows that $R(A+B) \subseteq N(B^{\#})=N(B)$. The other inclusion is similarly proved, upon taking the conjugate transposes of the first identity and using the fact that $(B^{\#})^*=(B^*)^{\#}$. \\
$(d)$ Consequence of $(c)$.\\
$(e)$ First, observe that $(A+B)^2=A^2+B^2+AB+BA=A^2-B^2$. We show that $N((A+B)^2)=N(A+B)$. Let $(A+B)^2x=0$ so that $(A^2-B^2)x=0$. Upon premultiplying by $B^{\#}$, one then obtains $$0=B^{\#}A^2x-B^{\#}B^2x=-B^{\#}BAx-Bx=-Ax-Bx,$$ i.e. $x \in N(A+B)$. \\
Let $A^2x=0$. Then $Ax \in N(A)\subseteq N(B)$ and so $BAx=0$ so that $B^2x=0$ (which in turn implies that $Bx=0$). Thus $(A+B)^2x=(A^2-B^2)x=0$ and so $(A+B)x=0$. Thus, $Ax=0$, proving that $N(A^2) \subseteq N(A)$, so that $A^{\#}$ exists. By $(c)$, $R(AB)=R(B^2)=R(B)$ and $N(AB)=N(B^2)=N(B)$. Since $B^{\#}$ exists, the subspaces $R(B)$ and $N(B)$ are complementary and so are $R(AB)$ and $N(AB)$, proving the existence of the group inverse of $AB$. \\
\end{proof}

\begin{rem}
Note that since $AB=BA$ and since the group inverse of a matrix is a polynomial in that matrix, the mutual commutativity relationships between $A, B, A^{\#}$ and $B^{\#}$ are applicable (note that the two conditions of the result above already imply $AB^{\#}=B^{\#}A$). This fact will be used frequently in our proofs.
\end{rem}

Applying Theorem \ref{prelimgrp} and Theorem \ref{grp}, we prove the next result.

\begin{thm}\label{intergrp}
Let $A, B$ satisfy the conditions of Theorem \ref{prelimgrp}. Then $$A^{\#}BA^{\#}=B^{\#}.$$
\end{thm}
\begin{proof}
Set $X=A^{\#}BA^{\#}.$ We must show that $X=B^{\#}$. \\
First, let $Bp=0$. Then $Xp=A^{\#}BA^{\#}p=A^{\#}A^{\#}Bp=0$. Conversely, let $Xp=0$ so that $A^{\#}BA^{\#}p=0$. Then $BA^{\#}p \in N(A) \subseteq N(B)$, and $BA^{\#}p \in R(B)$. Since $B^{\#}$ exists, this means that $BA^{\#}p=0$. So, $A^{\#}Bp=0$ so that $Bp \in N(A) \subseteq N(B)$ as well as $Bp \in R(B)$. So, $Bp=0$. We have shown that $Xp=0 \Longleftrightarrow Bp=0$. \\
Next, let $Bq=p,$ given $p, q\in R(B)$. Then $$Xp=A^{\#}BA^{\#}p=A^{\#}A^{\#}Bp.$$ Now, $p \in R(B) \subseteq N(A+B)$ and so $Bp=-Ap$. Thus $$Xp=-A^{\#}A^{\#}Ap=-A^{\#}p=-A^{\#}Bq.$$ Again, $q \in R(B) \subseteq N(A+B)$ and so $Bq=-Aq$. Thus, $Xp=A^{\#}Aq=q$, since $q \in R(B) \subseteq R(A)$. \\
Finally, let $Xp=q$ so that $$q=A^{\#}BA^{\#}p=A^{\#}A^{\#}Bp.$$ Now, $A^2(A^{\#})^2=AA^{\#}$ and since $R(B) \subseteq R(A)$, upon premultiplying the equation above by $A^2$, we then have $A^2q=Bp$. Then $BA^2q=B^2p$ and so $ABAq=-ABp$. This means that $p+Aq \in N(AB)=N(B)$ and so, $Bp+BAq=0$. Premultiplying by $B^{\#}$ and using the fact that $p, q \in R(B)$ as well as $B^{\#}BA=AB^{\#}B$, one obtains $p+Aq=0$. Premultiplying by $B^{\#}$ one obtains $B^{\#}Aq=-B^{\#}p$, and so $B^{\#}Bq=B^{\#}p.$ Premultiplying by $B$ and using the fact that $p \in R(B)$, we get $Bq=p$. Thus, one has $Xp=q \Longleftrightarrow Bq=p \ \ for ~all \ \ p,q \in R(B).$ By Theorem \ref{grp}, the conclusion follows.
\end{proof}

We are now in a position to prove the group inverse analogue of Theorem \ref{mainmp}.

\begin{thm}\label{maingrp}
Let $A,B \in \mathbb{C}^{n \times n}$ be related in such a way that 
$$AB^{\#}+BB^{\#}=0 \ \ and \ \ B^{\#}A+B^{\#}B=0.$$
We then have: $$(A+B)^{\#}=A^{\#} + B^{\#}  .$$
\end{thm}
\begin{proof}
Consider $$(A+B)(A^{\#}+B^{\#})=AA^{\#}+BB^{\#}+AB^{\#}+BA^{\#}=AA^{\#}+BA^{\#},$$
where we have used the fact that $BB^{\#}+AB^{\#}=0.$ This means that one has 
$$(A+B)(A^{\#}+B^{\#})(A+B)=A+AA^{\#}B+BA^{\#}A+BA^{\#}B.$$
Now, premultiplying ($(b)$ of Theorem \ref{prelimgrp} viz.,) $AB=-B^2$, by $B^{\#}$ one obtains $B=B^{\#}B^2=-B^{\#}AB$. Also, $BA^{\#}A=A^{\#}AB=B$ and so the sum of the last two terms in the expression above equals zero. We have shown that 
$$(A+B)(A^{\#}+B^{\#})(A+B)=A+B.$$ Next, from the first expression as above, one has
\begin{eqnarray*}
(A^{\#}+B^{\#})(A+B)(A^{\#}+B^{\#})&=&(A^{\#}+B^{\#})(AA^{\#}+BA^{\#})\\
&=& A^{\#}AA^{\#}+A^{\#}BA^{\#}+B^{\#}AA^{\#}+B^{\#}BA^{\#}\\
&=& A^{\#}+A^{\#}BA^{\#},
\end{eqnarray*}
where we have made use of the fact that $B^{\#}AA^{\#}+B^{\#}BA^{\#}=0$, since $B^{\#}A+B^{\#}B=0.$ By Theorem \ref{intergrp}, $A^{\#}BA^{\#}=B^{\#}$ and so the expression above simplifies to $A^{\#}+B^{\#}.$
Finally, 
\begin{eqnarray*}
(A^{\#}+B^{\#})(A+B)&=&A^{\#}A+A^{\#}B+B^{\#}A+B^{\#}B\\
&=& A^{\#}A+A^{\#}B,
\end{eqnarray*}
where we have made use of the identity $B^{\#}A+B^{\#}B=0.$ As was already mentioned, since $B$ and $A^{\#}$ commute, it also follows that 
$$(A^{\#}+B^{\#})(A+B)=(A+B)(A^{\#}+B^{\#}).$$ This completes the proof.
\end{proof}

\begin{rem}\label{where}
Where do the sufficient conditions for Theorem \ref{mainmp} and Theorem \ref{maingrp} come from? Recall that for $A,B \in \mathbb{C}^{m \times n}$, one says that $A {\leq}^* B$ (which is referred to as the ``star partial order'') if $AA^*=BA^*$ and $A^*A=A^*B$. Analogously, the notation $A {\leq}^{\#} B$ (which is referred to as the ``sharp partial order'') signifies the fact that $AA^{\#}=BA^{\#}$ and $A^{\#}A=A^{\#}B$ (assuming that the group inverse $A^{\#}$ exists). Pioneering contributions were made on matrix partial orders by Mitra \cite{mit}. There the author shows that if $A {\leq}^* B$, then one has the identity $(B-A)^{\dagger}=B^{\dagger}-A^{\dagger}$, while $(B-A)^{\#}=B^{\#}-A^{\#}$ holds if $A {\leq}^{\#} B$. It is now clear that the conditions of Theorem \ref{mainmp} are equivalent to the requirement that $-B {\leq}^* A$ (which therefore implies the identity in the title of this note), whereas the condition $-B {\leq}^{\#} A$ holds if and only if the hypothesis of Theorem \ref{maingrp} hold (which in turn, leads to the group inverse identity). However, the objective of this note is to divest the problem at hand from the notion of matrix partial orders, and also to present an independent and a self-contained treatment. It would be interesting to derive some characterizations for the two identities, studied in this note, to hold. 
\end{rem}

In what follows, we present a class of matrices that satisfy the conditions of Theorem \ref{mainmp} and Theorem \ref{maingrp}. First, we consider Theorem \ref{mainmp}.

\begin{ex}
We give a recursive procedure to compute matrices that satisfy the identity $AB^*+BB^*=0$ and $B^*A+B^*B=0.$ Let $a,b \in \mathbb{C}$ be such that $ab^*+bb^*=0$. Let $A, B \in \mathbb{C}^{2 \times 2}$ be defined by 
$$A=\begin{pmatrix} a & {\alpha}_1 \\ 
{\alpha}_2 & {\alpha}_3 \end{pmatrix} \ \ and  \ \ B=\begin{pmatrix} b & {\beta}_1 \\ 
{\beta}_2 & {\beta}_3 \end{pmatrix},$$
where ${\alpha}_i, {\beta}_i, ~i=1,2,3$ are to be determined. One may verify that $AB^*+BB^*=0$ translates into the following equations: 
\begin{center}
$({\alpha}_1+{\beta}_1){{\beta}_1}^*=0$\\ $(a+b){{\beta}_2}^*+({\alpha}_1+{\beta}_1){{\beta}_3}^*=0$ \\ $({\alpha}_2+{\beta}_2)b^*+({\alpha}_3+{\beta}_3) {{\beta}_1}^*=0$ \\ $({\alpha}_2+{\beta}_2) {{\beta}_2}^*+({\alpha}_3+{\beta}_3){ {\beta}_3}^*=0.$ 
\end{center}
Many choices are available and an easy option leads to the pair of matrices 
\begin{center}
$A=\begin{pmatrix} a & b \\ 
-b & {\alpha}_3 \end{pmatrix}$ \ \ and \ \ $B=\begin{pmatrix} b & -b \\ 
b & -{\alpha}_3 \end{pmatrix},$ 
\end{center}
where ${\alpha}_3$ is arbitrarily chosen. Then $A+B=\begin{pmatrix} a+b & 0 \\
0 & 0 \end{pmatrix}.$ In this case, one may verify that both $AB^*+BB^*=0$ and $B^*A+B^*B=0$. \\
Having constructed the basis step, one may now proceed to construct matrices with one extra row and column, satisfying the required identities, given a pair of matrices of lower order. More specifically, let $A, B \in \mathbb{C}^{m \times n}$ be such that $AB^*+BB^*=0$ and $B^*A+B^*B=0.$ Let $u_1 \in \mathbb{C}^m$ be chosen such that $(A+B)^*u_1=0$ and $v_1 \in \mathbb{C}^n$ be selected so that $(A+B)v_1=0$. Set $$M=\begin{pmatrix} A & u_1 \\
v_1^* & \alpha \end{pmatrix} \\ and ~\\ N=\begin{pmatrix} B & -u_1 \\
-v_1^* & -{\alpha} \end{pmatrix}.$$ Then $M, N \in \mathbb{C}^{(m+1) \times (n+1)}$ and one has $M+N=\begin{pmatrix} A+B & 0 \\
0  & 0 \end{pmatrix}.$ One may verify that $MN^*+NN^*=0$ and $N^*M+N^*N=0.$ There are more general choices for the matrices $M, N$ and we have given just one easy method of determining them.\\
Here is a numerical example: Let $a=1$ and $b=0$ so that $$A=\begin{pmatrix} 1 & 0 \\
0 & 1 \end{pmatrix} \\ and ~\\ B=\begin{pmatrix} 0 & 0 \\
0 & -1 \end{pmatrix}.$$ Then $A^{\dagger}=A^{-1}=I$ and $B^{\dagger}=B$ so that $A^{\dagger}+B^{\dagger}= \begin{pmatrix} 1 & 0 \\ 0 & 0 \end{pmatrix}.$ Also $A+B=\begin{pmatrix} 1 & 0 \\ 0 & 0 \end{pmatrix}$ and so $(A+B)^{\dagger}=A+B= A^{\dagger}+B^{\dagger}$.\\
Now, define $u_1=(0,2)^T$ and $v_1=(0,1)^T$ so that $(A+B)u_1=0=(A+B)^*v_1=0$. Define $$M=\begin{pmatrix} 1 & 0  & 0\\
0 & 1 & 1 \\0 & 2 & \alpha \end{pmatrix} \\ and ~\\ N=\begin{pmatrix} 0 & 0 & 0\\
0 & -1 & -1 \\0 & -2 & -\alpha \end{pmatrix}.$$ Then one may verify that $M^{\dagger}+N^{\dagger}=(M+N)^{\dagger}$.   
\end{ex}

Next, we construct matrices that satisfy the conditions of Theorem \ref{maingrp}. 

\begin{ex}
Let $a,b \in \mathbb{C}$ be such that $ab^{\#}+bb^{\#}=0$. Here $x^{\#}=\frac{1}{x}$, if $x \neq 0$ and $x^{\#}=0$, if $x=0$. Let $A, B \in \mathbb{C}^{2 \times 2}$ be defined by 
$A=\begin{pmatrix} a & {\alpha}_1 \\ 
{\alpha}_2 & {\alpha}_3 \end{pmatrix}$ and $B=\begin{pmatrix} b & {\beta}_1 \\ 
{\beta}_2 & {\beta}_3 \end{pmatrix},$
where ${\alpha}_i, {\beta}_i, ~i=1,2,3$ are to be determined. In order for $AB^{\#}+BB^{\#}=0$ to be satisfied, the said group inverse must exist.\\
Let us start with the case when $b \neq 0$, so that $a+b=0$ (and so $a \neq 0$). Since $B$ is singular, one has ${\beta}_3=\frac{{\beta}_1{\beta}_2}{b}.$ We have the following full rank factorization for $B$: 
$$B=FG, \ \ where \ \ F=\begin{pmatrix} 1\\ {\beta}_2 \end{pmatrix} \ \ and \ \ G=\begin{pmatrix} b & \frac{1}{a}{\beta}_1 \end{pmatrix}.$$  Also, since $B^{\#}$ exists, one must have $GF \neq 0$ and so $b+{\beta}_3\neq0$. Thus, $$B^{\#}=\frac{1}{(b+{\beta}_3)^2} \begin{pmatrix} b & {\beta}_1 \\ 
{\beta}_2 & {\beta}_3 \end{pmatrix}.$$ 
For the requirement $AB^{\#}+BB^{\#}=0$ to be satisfied, one has: 
\begin{center}
$({\alpha}_1+{\beta}_1){\beta}_2=0$ \\ $\frac{{\beta}_1 {\beta}_2}{b}({\alpha}_1+{\beta}_1)=0$ \\
$({\alpha}_2+{\beta}_2)b+({\alpha}_3+{\beta}_3) {\beta}_2=0$ \\
$({\alpha}_2+{\beta}_2){\beta}_1+\frac{{\beta}_1 {\beta}_2}{b}({\alpha}_3+{\beta}_3)=0.$
\end{center}
Here is a choice that leads to a nonzero $A+B$: ${\alpha}_2={\beta}_2=0$ so that ${\beta}_3=0$ and choose ${\alpha}_3=-b-{\alpha}_1$. Then
$B=\begin{pmatrix} b & b \\ 
0 & 0 \end{pmatrix}, A=\begin{pmatrix}  a & {\alpha}_1 \\ 
0  & -b-{\alpha}_1 \end{pmatrix}$ and $A+B=\begin{pmatrix}  a+b & {\alpha}_1+b \\ 
0  & -{\alpha}_1-b \end{pmatrix}.$ One may verify that $B^{\#}A+B^{\#}B=0$ also holds. \\

Next, let us consider the possibility when $b=0$, while $a$ is arbitrary and nonzero. One has $B=\begin{pmatrix} 0 & {\beta}_1 \\ 
{\beta}_2 & {\beta}_3 \end{pmatrix}.$ The possibility that ${\beta}_2={\beta}_3=0$ and ${\beta}_1\neq 0$ is ruled out, since $B$ would be nilpotent and so the group inverse does not exist. One needs to take into account two cases. \\
{\bf Case $(i)$}: ${\beta}_2\neq 0$. Then ${\beta}_1=0$ (otherwise $B$ would be nonsingular). Thus, $$B=FG, \ \ where \ \ F=\begin{pmatrix} 0\\ {\beta}_2 \end{pmatrix} \ \ and \ \ G=\begin{pmatrix} 1 & \frac{{\beta}_3}{2} \end{pmatrix}.$$ One must have $GF ={\beta}_3 \neq 0$ and so $$B^{\#}=\frac{1}{{{\beta}_3}^2} \begin{pmatrix}  0 & 0 \\ 
{\beta}_2 & {\beta}_3 \end{pmatrix}.$$ 
For the requirement $AB^{\#}+BB^{\#}=0$ to be satisfied, the following must hold: 
\begin{center}
${\alpha}_1{\beta}_2=0$ \\ ${\alpha}_1{\beta}_3=0$\\
$({\alpha}_3+{\beta}_3) {\beta}_2=0$ \\
$({\alpha}_3+{\beta}_3){\beta}_3=0.$
\end{center}
Now, since ${\beta}_2\neq 0$, one has ${\alpha}_1=0$ and ${\alpha}_3+{\beta}_3=0$ as well. Thus, one obtains the trivial situation, since $A+B=0$. \\
{\bf Case $(ii)$}: ${\beta}_2=0$. Then ${\beta}_3 \neq 0$. Arguing as above, one has $$B^{\#}=\frac{1}{{{\beta}_3}^2} \begin{pmatrix}  0 & {\beta}_1 \\ 
0 & {\beta}_3 \end{pmatrix}.$$
Again, for $AB^{\#}+BB^{\#}=0$ to hold, one must have the following: 
\begin{center}
$a{\beta}_1+({\alpha}_1+{\beta}_1){\beta}_3=0$\\
${\alpha}_2{\beta}_1 + ({\alpha}_3+{\beta}_3){\beta}_3=0.$
\end{center}
Imposing the condition $B^{\#}(A+B)=0$, in addition, one obtains:
\begin{center}
${\alpha}_2{\beta}_1={\alpha}_2{\beta}_3=0$\\
$({\alpha}_3+{\beta}_3){\beta}_1=({\alpha}_3+{\beta}_3){\beta}_3=0.$
\end{center}
So, ${\alpha}_2=0$ and ${\alpha}_3+{\beta}_3=0$. Assuming that ${\beta}_3 \neq -a$, we may choose 
${\beta}_1=-\frac{{\alpha}_1{\beta}_3}{a+{\beta}_3}$. By taking ${\alpha}_1=-{\beta}_1$, one may verify that all the six equations above hold. 
Thus, one has $A=\begin{pmatrix} a & -{\beta}_1 \\ 
0 & -{\beta}_3 \end{pmatrix}, B=\begin{pmatrix}  0 & {\beta}_1 \\ 
0  & {\beta}_3 \end{pmatrix}$ so that one has a nontrivial expression $A+B=\begin{pmatrix}  a+b & 0\\ 
0  & 0 \end{pmatrix}.$ 

Unlike the case of the Moore-Penrose inverse, there does not appear to be a recursive process to construct matrices satisfying the group inverse identity. Let us present a numerical illustration of the procedure above. Let $a=1$ and $b=0$. Define $A=\begin{pmatrix} 1 & -1 \\ 
0 & -1 \end{pmatrix}, B=\begin{pmatrix}  0 & 1 \\ 
0  & 1 \end{pmatrix}$ so that one has $A+B=\begin{pmatrix}  1 & 0\\ 
0  & 0 \end{pmatrix}.$ Then $(A+B)^{\#}=A+B$. Also, $$A^{\#}+B^{\#}=A^{-1}+B^{\#}=\begin{pmatrix}  1 & -1\\ 
0  & -1 \end{pmatrix} + \begin{pmatrix}  0 & 1\\ 
0  & 1 \end{pmatrix}=\begin{pmatrix}  1 & 0\\ 
0  & 0 \end{pmatrix}.$$
\end{ex}

\section{Concluding remarks}
We point to some directions for further study. We have considered the question of when a generalized inverse of a sum of two matrices equals the sum of their generalized inverses. As mentioned earlier, the problem of determining necessary and sufficient conditions for the two identities to hold, remains open. Next, one might be interested in asking the same question for sums involving three or more matrices. Apparently, an answer to that can turn out to be quite complex, considering the effort involved for the case of two matrices, as presented here. The second quest may be towards proving similar formulae for other clasical, as well new classes of generalized inverses, like the Drazin inverse, the core inverse or the Drazin-Moore-Penrose inverse. A third direction is to address the problem of determining when the studied identities hold, for elements in a ring. It is noteworthy that all the proofs presented here are linear algebraic and are free of multilinear notions like the rank or the determinant. Hence, with a little modification, they may be extended to the case of infinite dimensional spaces.

\end{document}